\documentclass[11pt]{amsart}
\usepackage{graphicx}
\usepackage{amsmath}
\usepackage{amssymb}
\usepackage[all]{xy}
\usepackage{hyperref}

\numberwithin{equation}{subsection}
\newtheorem{theorem}[subsection]{Theorem}
\newtheorem{classification-theorem}[subsection]{Classification Theorem}
\newtheorem{decomposition-theorem}[subsection]{Decomposition Theorem}
\newtheorem{proposition-definition}[subsection]{Proposition-Definition}
\newtheorem{periodicity-conjecture}[subsection]{Periodicity Conjecture}
\newtheorem{lemma}[subsection]{Lemma}
\newtheorem{proposition}[subsection]{Proposition}
\newtheorem{definition}[subsection]{Definition}
\newtheorem{corollary}[subsection]{Corollary}

\newtheorem{assumption}[subsection]{Assumption}
\newtheorem{example}[subsection]{Example}
\newtheorem{remark}[subsection]{Remark}

\newcommand{\rad}{\operatorname{rad}\nolimits}
\newcommand{\res}{\operatorname{res}\nolimits}

\newcommand{\Hom}{\operatorname{Hom}\nolimits}

\newcommand{\cok}{\operatorname{cok}\nolimits}

\newcommand{\Ext}{\operatorname{Ext}\nolimits}

\renewcommand{\mod}{\operatorname{mod}\nolimits}

\newcommand{\rep}{\operatorname{rep}\nolimits}
\newcommand{\proj}{\operatorname{proj}\nolimits}

\newcommand{\ind}{\operatorname{ind}\nolimits}

\newcommand{\Z}{\operatorname{\mathbb{Z}}\nolimits}
\newcommand{\C}{\operatorname{\mathbb{C}}\nolimits}
\newcommand{\N}{\operatorname{\mathbb{N}}\nolimits}

\newcommand{\F}{\operatorname{\mathbb{F}}\nolimits}

\newcommand{\gdim}{\operatorname{\underline{dim}}\nolimits}

\newcommand{\gpr}{\operatorname{gpr}\nolimits}

\newcommand{\iso}{\stackrel{_\sim}{\rightarrow}}

\newcommand{\id}{\mathbf{1}}

\newcommand{\cc}{{\mathcal C}}
\newcommand{\cd}{{\mathcal D}}
\newcommand{\ce}{{\mathcal E}}

\newcommand{\cg}{{\mathcal G}}
\newcommand{\ch}{{\mathcal H}}

\newcommand{\cl}{{\mathcal L}}
\newcommand{\cm}{{\mathcal M}}
\newcommand{\cn}{{\mathcal N}}

\newcommand{\cp}{{\mathcal P}}
\newcommand{\cR}{{\mathcal R}}

\newcommand{\cS}{{\mathcal S}}

\renewcommand{\tilde}[1]{\widetilde{#1}}

\begin{document}
\title[Quiver varieties and Hall algebras]{Quiver varieties and Hall algebras}
\author{Sarah Scherotzke and Nicolo Sibilla}
\address{S.~S.~: University of Bonn, Mathematisches Institut, 
Endenicher Allee 60, 53115 Bonn, Germany}

\email{sarah@math.uni-bonn.de}

\keywords{Nakajima quiver varieties, Hall algebra}
\subjclass[2010]{13F60, 16G70, 18E30}

\begin{abstract}
In this paper we give a  geometric construction of the quantum group $U_t(\cg)$ using Nakajima categories, which were developed in \cite{Scherotzke}. Our methods allow us to establish a direct connection between the algebraic realization of the quantum group  as Hall algebra by Bridgeland \cite{Bridgeland13} and its geometric counterpart 
by Qin \cite{Qin14}.

\end{abstract}

\maketitle

\tableofcontents
\section{Introduction}

Let $Q$ be a Dynkin quiver of type $A, D, E$. 
Ringel showed in \cite{Ringel90} that the positive part of the quantum group $U_{t}(\cn^+)$ specialized at $t=\sqrt{q}$ can be realized as a the Ringel-Hall algebra to the abelian category $\F_q Q$--mod. An alternative geometric construction of $U_{t}(\cn^+)$ was given by Lusztig \cite{Lusztig90}, \cite{Lusztig91} who considers a convolution product of shifted perverse sheaves on the moduli spaces of representations of $\C Q$.  This approach provides us with a natural basis of $U_{t}(\cn^+)$ corresponding to perverse sheaves which is called the canonical basis or crystal basis \cite{Kashiwara91}.  Recently, Hernandez-Leclerc \cite{HernandezLeclerc11} gave a realization of the positive part of the quantum group using graded Nakajima quiver varieties.

In 2013, Bridgeland  \cite{Bridgeland13} realized $ U_{t}(\cg)$ as a quotient of the semi-derived Hall algebra of the exact category of $2$-cyclic complexes $Comp_{\Z/2}( \proj \F_q Q)$ of projective $\F_q Q$--modules. More generally, semi-derived  Hall algebras of Frobenius categories have been studied by Gorsky in \cite{Gorsky}. They are localizations of the classical Hall algebra with respect to projective-injective objects. 
As Ringel's Hall algebra appears naturally as a sub algebra of the semi-derived Hall algebra to $Comp_{\Z/2}( \proj \F_q Q)$, we can view Bridgeland's construction as a natural extension of Ringel's results.

Recently, Fan Qin constructed a geometric counterpart using the dual graded Grothendieck ring of constructible sheaves on cyclic affine Nakajima quiver varieties \cite{Qin14}. This Hall algebra construction is a natural extension of Hernandez and Leclerc's construction. In this paper we propose a new geometric construction of the quantum group.  Our techniques allow us to establish a direct comparison between the algebraic realization of the quantum group $U_t(\cg)$ by Bridgeland \cite{Bridgeland13} and its geometric counterpart proposed by Qin \cite{Qin14}. The main ingredient is the theory of Nakajima categories developed in \cite{KellerScherotzke13a} and \cite{Scherotzke}. 
 We proved in \cite{Scherotzke} that 
affine quiver varieties, usually defined as colimits of GIT quotients, can be described as moduli spaces of representations of the singular Nakajima category $\cS$. See also \cite{LeclercPlamondon12} and \cite{KellerScherotzke13a} for the analogous result in the graded case. 

We show that the exact Frobenius category $Comp_{\Z/2} (\proj kQ)$ used by Bridgeland can be realized as the category of Gorenstein projective modules of a singular Nakajima category $\cS$, denoted $\gpr \cS$. 
More precisely, we obtain the following result.

\begin{theorem} [Theorem \ref{nakajima category}]
\label{main2}
 There is a singular Nakajima category $\cS$ which yields an isomorphism $\psi: Comp_{\Z/2} (\proj kQ) \to \gpr \cS$ of exact Frobenius categories. The map $\psi$ establishes a bijection between the acyclic complexes and the projective $\cS$--modules. Furthermore, the moduli space $\rep(w, \cS)$ is isomorphic to $\cm_0(w)$, the cyclic affine quiver variety. 
 \end{theorem}

Bridgeland's result together with the above Theorem allows us to identify the generators of the quantum group with isoclasses of $\cS$--modules. Each isoclass  corresponds to a stratum of the Nakajima quiver variety. 
We consider the IC complexes on these strata (\emph{Bridgeland strata}).  
They give classes in the dual graded Grothendieck group of the cyclic affine quiver variety. One of our main results is that the subalgebra generated by these classes, denoted $\cS \ch_{\gpr \cS}$, provides a new geometric construction of the quantum group.

Our  construction is different from Qin's  \cite{Qin14}. Bridgeland strata 
do not coincide with the strata he considers, instead we show in Corollary \ref{transversal generator} that they are transversal to each other.  In Theorem \ref{iso} we prove that, after localization along the IC complexes on minimal strata, the algebra  considered by Qin and $\cS \ch_{\gpr \cS}$ are isomorphic. 

\begin{theorem}[Theorem \ref{iso}]
\label{main1}
The graded dual Grothendieck ring $K^{* gr}(\gpr \cS)$ admits a Hall algebra structure. Further its subalgebra $
\cS \ch_{\gpr \cS}$  
is isomorphic to Qin's algebra (after localization along the IC complexes on minimal strata). 
\end{theorem}

We refer to Theorem \ref{iso} for the exact statement, which is stronger and requires setting up some notation. 
 In particular this isomophism preserves the generators of the quantum group $E_i$, $F_i$ and $H_i$ chosen by Bridgeland and Qin respectively. 

Further our result clarifies the role played by the localizations considered by Bridgeland and Qin. We show that the localizing sets they consider correspond to each other across the dictionary provided by Theorem  \ref{main2} and Theorem \ref{main1}.
Indeed, call a stratum minimal if it vanishes under the stratification functor constructed in \cite{Scherotzke} and \cite{KellerScherotzke13a}. Qin's construction involves localizing at the minimal strata. 
The minimal strata in the subvariety of Gorenstein projective $\cS$ modules correpond under $\phi$ 
to the projective-injective modules, whose isomorphism classes are localized in the semi-derived Hall algebra construction of Bridgeland.  

We also prove the existence of Hall polynomials for a class Nakajima categories. This implies that we can meaningfully consider  their generic Hall algebra (see also \cite{Ringel90a}, \cite{ChenDeng}). 
\begin{theorem}[Theorem \ref{hall}]
Under the assumption \ref{assumption}, the exact Frobenius category $\gpr \cS$ satisfies the Hall polynomial property. 
\end{theorem}
Furthermore, we observe in Proposition \ref{pi} that the graded Grothendieck ring for generalized quiver varieties introduced in \cite{Scherotzke} yields a Hall coalgebra structure.

{\bf Aknowledgements:} We would like to thank Bernhard Keller for his constant encouragement. We are also grateful to Fan Qin, Mikhail Gorsky, Jan Schr{\"o}er, Markus Reineke and Olivier Schiffmann 
for useful conversations. Furthermore, we would like to thank David Hernandez for providing references.

\section{Quiver Varieties and Nakajima Categories}

\subsection{Notation}For later use, we introduce the following notations. Let $k$ be a field and let $\mod k$ be the category of finite-dimensional $k$-vector spaces. If $\cc$ is a $k$-linear category whose $\Hom$-spaces are finite-dimensional we denote 
$\mod(\cc)$ the category of finite-dimensional {\em right $\cc$-modules}, meaning $k$-linear functors 
$\cc^{op} \to \mod(k)$. We write $\cc(u,v)$ for the space of morphisms $\Hom_\cc(u,v)$ and
$D$ for the duality over the ground field $k$. 
For each object $x$ of $\cc$, we have the {\em indecomposable projective
module}
\[
x^{\wedge}=x^{\wedge}_\cc = \cc(?,x): \cc^{op} \to \mod k
\]
and the {\em indecomposable injective module}
\[
x^{\vee}= x^\vee_\cc = D (\cc(x,?)): \cc^{op} \to \mod k.
\]
We denote $\proj \cc$ the full additive subcategory of $\mod(\cc)$ spanned by indecomposable projective modules.  We denote by $\cc_0$ the set of objects of $\cc$ and mean by a dimension vector of $\cc$ a function $w: \cc_0 \to \N$ with
finite support. We define $\rep( w, \cc)$ to be the variety of $\cc$-modules $M$ such that $M(u) = k^{w(u)}$ for each
object $u$ in $\cc_0$.  

Throughout the paper we let $Q$ be a Dynkin quiver of type $A, D, E$ with set of vertices $Q_0$. We will denote by $S_i$, $P_i$ and $I_i$ respectively the simple, the indecomposable projective and the indecomposable injective $kQ$--modules associated with every $i \in Q_0$. We will denote by $\cd_Q$ the derived category of finite-dimensional $kQ$--modules, by $\tau$ its Auslander-Reiten translation and by $\Sigma$ the shift functor. 

\subsection{Nakajima categories}\label{Nakajima categories}

In this section we recall briefly the definition of regular and the singular Nakajima categories $\cR$ and $\cS$ as defined in \cite{Scherotzke}. 
Let $\Z Q$ be the repetition quiver of $Q$. As $Q$ is of Dynkin type, the Auslander-Reiten quiver of $\cd_Q$
coincides with $\Z Q$. 

Furthermore by  Happel's Theorems \cite{Happel87} and \cite{Happel88}, the mesh category $k(\Z Q)$ is equivalent to $\ind \cd_Q$, the category of indecomposable objects of $\cd_Q$. Using this isomorphism, we label once and for all 
the vertices of $\Z Q$ by the isomorphism classes of 
indecomposable objects of $\cd_Q$. 
The Nakajima category depends now on the choice of a triangulated automorphism $F: \cd_Q \to \cd_Q$ and an $F$-invariant subset $C$ of isomorphism classes of objects in  $\ind \cd_Q$, satisfying the following: 

\begin{assumption}
The orbit category $\cp:= \cd_Q/F$ is triangulated and the canonical projection $\cd_Q \to \cd_Q/F$ is a triangulated functor. 
Furthermore, we assume that for every $x \in \cd_Q$ there are objects $c,\ c'\in C$ such that ${\cd_Q}(x,c)\not =0$ and $\cd_Q (c',x) \not =0$. If $C$ has this property we call it a \emph{configuration}.
\end{assumption}
We consider the quiver $\Z Q_C$ which we obtain by adding to every vertex $c \in C$ an object $\sigma(c)$ together with an arrow $c\to \sigma(c)$ and an arrow $\sigma(c) \to \tau^{-1} c$.
The functor $F$ can then be lifted uniquely to an automorphism of $\Z Q_C$.
Next, we consider the quiver $\Z Q_C/F$ obtained by merging $F$-orbits into a single vertex. Note that by our assumptions, $\ind \cp$ is equivalent to the mesh category 
of $\Z Q/F$. 

\begin{example}\label{exa1}
Consider the quiver $Q=A_2:  2 \to 1$, the configuration given by $C=\{   \tau^i S_1 \text{ for all }i \in \Z  \}$,  and the functor $F=\Sigma^2 $. Then $\Z Q_C/F$ is the following finite quiver

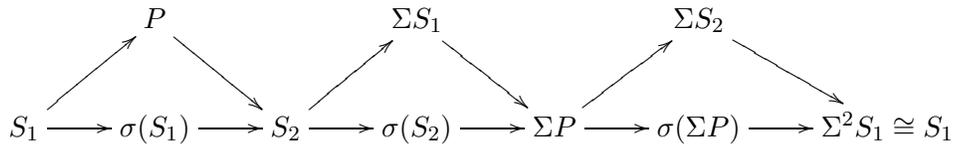
\[  \xymatrix{ 
& P \ar[rd] & & \Sigma S_1 \ar[rd] & & \Sigma S_2  \ar[rd] &  \\
S_1 \ar[ur]  \ar[r] & \sigma(S_1) \ar[r] & S_2 \ar[r]  \ar[ru] & \sigma(S_2)  \ar[r] & \Sigma P \ar[ru] \ar[r] &\sigma(\Sigma P)  \ar[r] & \Sigma^2 S_1  \cong S_1 } 
\]
  
Here the configuration is minimal in the sense that there is no proper  subset 
of $C$ giving rise to a configuration. 
\end{example}
Note that every triangulated functor $\cd_Q \to \cd_Q$ which commutes with $F$ defines canonically an automorphism of the quiver $\Z Q/F$. In particular, the Auslander-Reiten translation $\tau$ 
and $\Sigma$ define automorphisms on $\Z Q/F$, which we denote by the same symbols.

The \emph{regular Nakajima category} $\cR$ associated to $F$ and $C$ is given by the $k$-linear category with objects the vertices of $\Z Q_C/F$. Morphisms $\cR(x,y)$ are $k$-linear combinations of paths from $x$ to $y$ modulo the subspace spanned by all elements $u r_z v$, where
$u$ and $v$ are paths and $r_z$ is the sum of all paths from $\tau(z)$ to $z$ where $z$ is a vertex of the subquiver $\Z Q/ F$. 
Note that the vertices of $\Z Q/ F$ corresponds to the isomorphism classes of objects in $\ind \cp$.

The {\em singular graded Nakajima category
$\cS$} is the full subcategory of $\cR$ whose objects are vertices $\sigma(c)$ for all $c \in C$. 

The graded Nakajima categories $\cR^{gr}$ and $\cS^{gr}$ are defined analogously, replacing $\Z Q_C/F$ by $\Z Q_C$.
We refer to \cite{KellerScherotzke13a} and \cite{LeclercPlamondon12} for more details on the graded case. 

\begin{assumption}
Throughout this paper, we will consider $\Hom$--finite Nakajima categories $\cR$. 
\end{assumption}

By definition, there is a restriction functor $\res: \mod \cR \to \mod \cS$, which is exact.

In this paper, we will focus on the category of Gorenstein projective $\cS$--modules, denoted $\gpr \cS$ (see \cite{Scherotzke}, \cite{Najera}). 
This is a full and exact subcategory of $\mod \cS$ which is Frobenius and closed under extensions. The class of projective-injective objects is given by the projective $\cS$--modules. 
This category is related to $\cR$ by the following result.

\begin{theorem}\label{equivalence}
The restriction functor induces an equivalence of exact Frobenius categories $\res: \proj \cR \to \gpr \cS$. 
\end{theorem}

The category $\cS$ can be described as a path algebra $k Q_{\cS}$ modulo an admissible ideal $I$. We refer to Proposition 3.16 \cite{Scherotzke} for the description of the quiver $Q_{\cS}$ and the admissible ideal $I$.
\begin{example}
Let us choose $Q$, $C$ and $F$ as in the previous example \ref{exa1}. It is easy to see that $\cR$ is $\Hom$-finite.  In this example the quiver $Q_\cS$ is the following
 
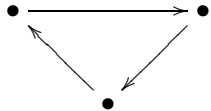
\[  \xymatrix{ 
\bullet \ar[rr]  &  & \bullet   \ar[ld] \\
&  \ar[lu]  \bullet & } 
\]  
and $I$ is the ideal generated by all paths of length two. 
As $\mod \cS$ is self-injective, the category $\gpr \cS$ is equal to $\mod \cS$ and in particular abelian. 
\end{example}

\subsection{Nakajima quiver varieties}

Quiver varieties were introduced by Nakajima \cite{Nakajima94} and \cite{Nakajima98}. 
In \cite{Scherotzke} we introduce generalized quiver varieties using Nakajima categories. Our approach via Nakajima categories also provides  alternative constructions of classical quiver varieties, and is essential to compare Bridgeland's and Qin's work  on the quantum group. 
We will give a short introduction to generalized quiver varieties here.  

We denote by $S_x$ the simple one-dimensional $\cR$--module associated to a vertex $x \in \cR_0$. 
Recall that an $\cR$-module $M$ is {\em stable} if $\Hom_\cR(S_x, M)$ vanishes 
and $M$ {\em costable} if we
have $\Hom_\cR(M, S_x)=0$ for all simple modules $S_x$ supported at $ x\in \cR_0-\cS_0$.
A module is {\em bistable} if it is both stable and costable. 

For a choice of two dimension vectors $w: \cS_0 \to \N$ and $v: \cR_0-\cS_0 \to \N$, we define 
the \emph{affine quiver variety} $$\cm_0^{reg}(w):= \rep(w, \cS)$$ to be the moduli space of representations of $\cS$ with dimension vector $w$
and the smooth quiver variety to be $$\cm(v,w):= \rep(v,w, \cR)^{stable}/\prod_{i \in \cR_0-\cS_0} Gl (v(i)), $$ the quotient of the moduli space of space of stable representations by the the general linear group acting via base change in all vertices of $\cR_0-\cS_0$. 

As shown in Theorem 3.1 \cite{Scherotzke} for the general case and in \cite{Nakajima94} for the classical case, the varieties $\cm(v,w)$ are smooth and quasi-projective. Furthermore, 
we have a proper surjective map 
$$\pi: \sqcup \cm(v,w) \to \cm_0(w)$$ inducing a stratification $$\cm_0(w)=\sqcup\cm_0^{reg}(v,w)$$ into finitely many non-empty 
smooth locally closed strata $\cm_0^{reg}(v,w).$
We show in Theorem 3.5 of \cite{Scherotzke} that for $F=\tau^n$ one recovers Nakajima's $n$-cyclic quiver varieties.

\subsection{Stratification functor} \label{stratification}

The restriction functor $$\res: \mod \cR \to \mod \cS$$ is 
a localization functor and admits a right and a left adjoint which we denote $K_R$ and $K_L$ respectively. 
We define the intermediate extension $$K_{LR}: \mod \cS \to \mod \cR$$ as the image of the canonical map 
$K_L \to K_R$ (see \cite{KellerScherotzke13b} for general properties).  

The next lemma establishes the connection of the statification of $\cm_0(w)$ with the intermediate extensions. 
\begin{lemma}\label{lemma:stable}
An $\cR$ module $M$ is bistable if $K_{LR} \res M \cong M$. Furthermore, we have $N\in \cm^{reg}_0(v,w)$ 
if and only if $\gdim K_{LR} N =(v,w)$. 
\end{lemma} 
It will be useful to introduce a related functor $CK$ defined by $CK(M) :=\cok(K_{LR}(M) \to K_R(M))$ for all $M \in \mod \cS$. 
Now we have an obvious isomorphism
$\cR/ \langle \cS \rangle \iso \cp$ where $\langle \cS \rangle$ denotes
the ideal generated by the identical   morphisms of $\cS$. Therefore,
we may view $CK(M)$ as a $\cp$--module. 

We have shown in \cite{Scherotzke} that the image of $CK$ consists of projective modules and we obtain the following result. 
\begin{theorem}
The functor $CK: \mod \cS \to \proj \cp, M \mapsto CK(M)$ maps objects in the same stratum to isomorphic representable objects.
\end{theorem}

Hence we call $CK$ the \emph{stratification functor}. Furthermore, we call a stratum \emph{minimal} if $CK$ takes value zero on the stratum. We call a pair of strata  \emph{transversal}, if they have isomorphic images under $CK$. We refer to 3.3.2 of \cite{Nakajima01} for the geometric definition in the case of classical quiver varieties.

\section{Hall algebras of Nakajima Categories}
In his paper \cite{Bridgeland13}  Bridgeland shows that a quotient of the twisted semi-derived Hall algebra of the category of $2$-periodic complexes $Comp_{\Z/2}(\proj \F_q Q)$ is isomorphic to $U_{t}(\cg)$. We refer to section 2.1 of \cite{Bridgeland13} for the presentation of $U_t(\cg)$ used by Bridgeland.  

Recall that for an exact category $\ce$ the category of $\Z/n$-periodic complexes has objects the $n$-cyclic complexes   
$$
\xymatrix{
  &  P_i  \ar@/^1pc /[dr]^{d_i}  & \\
 P_{i-1} \ar@/^ 1pc/[ur]^ {d_{i-1}} & & P_{i+1} \ar@/^1pc /[dl]^{d_{i+1}} \\
  & \cdots \ar@/^1pc /[lu]^{d_{i-2}} & }
$$
  with $P_i \in \ce$, $d_i \in \ce(P_i , P_{i+1})$ such that  $d_i d_{i+1} =0$  for all $ i \in \Z/n$ and morphisms given by 
family of maps $f_i : P_i \to P_i'$ that commute with the differential $d_i$.

We show next that the exact categories $Comp_{\Z/2}(\proj \F_q Q)$ are in fact equivalent to a Nakajima category as introduced in Section \ref{Nakajima categories}. 

\begin{theorem}\label{nakajima category}
Let $k$ be any field. The category $Comp_{\Z/n}(\proj kQ)$ of n-cyclic complexes is equivalent to $\proj \cR$ for the regular Nakajima category $ \cR$ corresponding to:
\begin{itemize}
\item the functor $F=\Sigma^n$,  
\item the configuration $C$ given by the vertices labelled by $\Sigma^j  S_i$ for all $i \in Q_0$ and $j \in \Z$. 
\end{itemize}
\end{theorem}
\begin{proof}
As shown in \cite{ChenDeng}, the exact category $Comp_{\Z/n}(\proj kQ)$ admits Auslander-Reiten sequences and using Proposition 2.5 in \cite{ChenDeng},
we see that the conditions of Lemma 3 in \cite{Ringel84} are satisfied. It follows that $Comp_{\Z/n}(\proj kQ)$ is a standard category, that is $Comp_{\Z/n}(\proj kQ)$
is equivalent to the mesh category of its Auslander-Reiten quiver. Furthermore, $Comp_{\Z/n}(\proj kQ)$ is a Frobenius model for  the orbit category 
$\cd_Q/\Sigma^n$ meaning $Comp_{\Z/n}(\proj kQ)$ is a Frobenius category with projective-injective objects the acyclic complexes and its stable category is equivalent to 
$\cd_Q/\Sigma^n$. As $Comp_{\Z/n}(\proj kQ)$  is standard, we know by Theorem 5.6 of \cite{Scherotzke}, that the Frobenius category $Comp_{\Z/n}(\proj kQ)$ is equivalent to $\proj \cR$ for the functor $F=\Sigma^n$. 
By Theorem 5.6 \cite{Scherotzke}, every $c\in C$ corresponds to the position 
of a complex such that a projective-injective object appears as a direct summand in the middle term of its Auslander-Reiten sequence. 
The projective-injective modules in $Comp_{\Z/n}(\proj kQ) $ are given by shifts of the $n$-cyclic complexes 
$$\cdots \to  0 \to P_i \stackrel{\id}  \to P_i \to 0  \to \cdots $$ for all $i$ in $ Q_0$. These appear in the middle term of the Auslander-Reiten sequences starting in shifts of complexes $$\cdots \to  0 \to \rad P_i \to P_i \to 0 \to \cdots.$$ In the stable category $\underline{Comp}_{\Z/n}(\proj kQ) $, the $j$--shift of the complex $$ \cdots \to  0 \to \rad P_i \to P_i \to 0 \to \cdots $$ corresponds under 
the isomorphism $\underline{Comp}_{\Z/n}(\proj kQ) \cong \cd_Q/ \Sigma^n $ to its homology which is $\Sigma^j S_i$. 
\end{proof}
Let us denote by $\psi: Comp_{\Z/n}(\proj kQ) \to \proj \cR \to  \gpr \cS$ the composition of isomorphisms of exact Frobenius categories. Under this isomorphism, we have

\begin{align*}
&\psi( \rad P_i \hookrightarrow P_i)  \mapsto \res   \Sigma S_i^{\wedge},\\
&\psi( P_i  \hookleftarrow \rad P_i)   \mapsto \res  S_i^{\wedge},\\
& \psi( P_i \stackrel{\id} \rightarrow P_i)  \mapsto \sigma( \Sigma S_i)^{\wedge},\\
& \psi( P_i \stackrel{\id} \leftarrow P_i)  \mapsto \sigma( S_i)^{\wedge}. 
\end{align*}

\begin{example}\label{exa2}
We consider the Nakajima category for $Q=A_2: 1 \to 2 $, $F=\Sigma^2$ and $C=\{ \Sigma ^i S_1, \Sigma^iS_2, \text{ for all } i \in Z \}$. 

\[  \xymatrix{ 
& P \ar[rd] & & \Sigma S_1 \ar[rd] \ar[r] & \sigma(\Sigma S_1)  \ar[r] & \Sigma S_2 \ar[r]  \ar[rd] & \sigma(\Sigma S_2)  \ar[r] & \Sigma^2 P \cong P \\
S_1 \ar[ur]  \ar[r] & \sigma(S_1)  \ar[r] & S_2 \ar[r]  \ar[ru] & \sigma(S_2)  \ar[r] & \Sigma P \ar[ru] && \Sigma^2 S_1  \cong S_1 \ar[ru] & } 
\]
Here the configuration is minimal in the sense that there is no proper  subset 
of $C$ yielding a configuration. Note that the configuration gives rise to $\Hom$-finite Nakajima categories $\cS$ and $\cR$. As an algebra $\cS$ is given by the path algebra of the following quiver, $Q_\cS$ 
 \[  \xymatrix{ 
\sigma(\Sigma S_1) \ar[r]^{a} \ar@/^/[d]^{\alpha'} &  \sigma(\Sigma S_2)   \ar@/^/[d]^{\beta} \\
\sigma(S_1) \ar[u]^{\alpha}  \ar[r]^b & \sigma(S_2) \ar[u]^{\beta'} } 
\]  subject to the relations 
$\langle \alpha \alpha', \  \alpha' \alpha,\  \beta \beta', \ \beta' \beta,\   \alpha a-b \beta' , \ \alpha' b-a \beta \rangle $ (recall that in this example $\cS$ is a subcategory of $\cR$ with four objects: in the diagram above the vertices of $Q_\cS$ are denoted like the corresponding objects of $\cS$).   
Clearly $\mod \cS$ is not self-injective and therefore $\gpr \cS$ is a full subcategory of $\mod \cS$. 
By Theorems \ref{nakajima category} and \ref{equivalence}, the Frobenius category $\gpr \cS$ is equivalent to the category $Comp_{\Z/ 2}( \proj k Q)$ 
considered by Bridgeland. 
\end{example}
We consider the Hall algebra associated to $\proj \cR$, where $k=\F_q$ the finite field $q$ elements and $q$ a prime power. 
The morphism space $${\proj \cR}(x^{\wedge}, y^{\wedge}) \cong {\cR}(x,y)$$ is given by all 
$\F_q$-linear combinations of paths from $x$ to $y$ 
in the quiver of $ \Z Q_C/F$ modulo the mesh relations which are monomial. 
Hence we have that, for all $M, N \in \proj \cR$, the dimension of $\Hom(M, N)$ does not depend on $q$. 
Thus $|\Hom(M, N)|$ is given 
by $q^n$, where $n$ is the dimension of $\Hom(M, N)$.  

\begin{assumption}\label{assumption}
We assume that $F$ satisfies the following property: Let $M$ and $N$ be indecomposable objects in $\cd_Q$, then
$ \Ext^1(M, F^iN) $ vanishes for all but at most one $i \in \Z$. 
\end{assumption} 
One easily verifies that the above condition on $F$ is satisfied for $F=\Sigma^n$ for all $ n>1$, but fails for $n=1$. 

\begin{lemma}
The above assumption holds for $F= \Sigma \tau^{-1}$ and any field $k$. 
\end{lemma}
\begin{proof}
Let $M$ and $ N$ be indecomposable objects in $\cd_Q$. We assume without loss of generality that $M $ is a projective $kQ$--module and that $\cd_Q(M, N)$ does not vanish. Then $N\in \mod kQ$ and ${\cd_Q}(M, F^iN)= {\cd_Q}(M, \tau^{-i}\Sigma^iN ) $ vanishes for all $i \ge 1$ as $M$ is projective and $\tau^{-i} N$ is a positive shift of a $kQ$--module. For $i <0$, we see that $ {\cd_Q}(M, \tau^{-i}\Sigma^i N ) $ vanishes  as $\tau^{-i} N $ is a negative shift of a $kQ$--module. 
\end{proof}

Let $M, N$ and $L$ be in $ \proj \cR$. The Hall number  $F_{N,M}^L$ is given by the quotient of $|\Ext^1(M, N)_L|$ by $|\Hom(M,N)|$, where $\Ext^1(M, N)_L$ is the set of isomorphism classes of exact sequences 
$$ 0 \to N \to L \to M \to 0.$$

The Hall polynomial property for $Comp_{\Z/n}(\proj kQ)$ and $n>1$ has also been shown by Chen and Deng in \cite{ChenDeng}. Our proof is similar. 
\begin{theorem}\label{hall}
The Frobenius category $\proj \cR$ over the field $k=\F_q$ satisfies the Hall polynomial property, that is the Hall number $F_{N,M}^L$ is given by a Laurent polynomial in $q$. 
\end{theorem}
\begin{proof} 
As we discussed $|\Hom(M, N)|$ is given by $q^n$ for some $n$. Thus it only remains to show that for all $M, N, L \in \proj  \cR$ the number of elements in $\Ext^1(M, N)_L$ is given by a polynomial in $q$.  
The Hall polynomial property is satisfied for $\proj \cR^{gr}$ as $\proj \cR^{gr}$ is representation directed, this can be proved using the same argument as in \cite{Ringel90a}. 
Recall that the category $\proj  \cR$ is an orbit category of $\proj \cR_C^{gr}$ by the functor $F$. We first check the Hall Polynomial property 
in the case that $M$ or $N$ is indecomposable. As the objects $\sigma(x)^{\wedge}$ are projective-injective in $\proj \cR$, we assume without loss of generality that there are objects 
$x, y_i \in \cR_0- \cS_0$ 
such that $M \cong x^{\wedge}$ 
and $ N \cong \bigoplus_{i=1}^n y_i^{\wedge}$. By Assumption \ref{assumption}, $\Ext^1_{\cd_Q}( x,F^j y_i) $ vanishes for all $j \not =0$.

Hence, 
$$ 
\Ext^1_{\proj \cR}(x^{\wedge}, N)\cong \Ext^1_{ \cp} (x,\bigoplus_{1 \le i \le n}  y_i) \cong 
$$

$$ \cong \bigoplus_{j \in \Z} \Ext^1_{\cd_Q}(x, \bigoplus_{1 \le i \le n} F^j(y_i)) \cong \Ext^1_{\cd_Q}(x, \bigoplus_{1 \le i \le n} y_i).$$ 
It follows that $\Ext^1(x^{\wedge}, N)_L \cong \sqcup_{j \in \Z} \Ext^1_{\cd_Q}( x, \bigoplus_{1 \le i \le n} y_i)_{F^j(L)} $ and as the number on the right hand side is given by a polynomial in $q$, the same holds for the left hand side. 
This proves the existence of Hall polynomials if $M$ is indecomposable. The case when $N$ is indecomposable is proved analogously.  

Now suppose that $0 \to A \to M \to B \to 0 $ is an exact sequence in $\gpr \cS \cong \proj \cR$, then we have that $\dim \Ext^1(M, M) \le \dim \Ext^1( A \oplus B, A \oplus B) $. Furthermore, by 
Lemma 2.1 of \cite{Guo} we have equality of dimensions if and only if the exact sequence is split. 
This result allows us to proceed as in \cite{ChenDeng} and prove the claim when $M$ is decomposable. 
\end{proof}

The category $\proj \cR$ viewed over $\F_q$ is exact and finitary. Therefore its Ringel-Hall algebra, which we denote by $H_q(\proj \cR) $ is well-defined by \cite{Hubery}. This algebra is generated by the isomorphism classes of objects $M \in \proj \cR$, which we denote $[M] \in H_q(\proj \cR)$.  As $\proj \cR$ satisfies the Hall polynomial property, we can view $q$ as a parameter and consider the generic Hall algebra, which we also denote $H_q(\proj \cR)\cong H_q(\gpr \cS)$. We denote $H^{tw}_q$ the Hall algebra whose multiplication is twisted by an Euler form defined as in \cite{Bridgeland13}. 

We denote  
$SH_q:= H_q(\proj \cR)[ \sigma(S_i)^{\wedge} , \sigma(\Sigma S_i)^{\wedge}]^{-1} $ 
the \emph{semi-derived Hall algebra} of $\proj \cR$, 
which is the localization of 
$H_q(\proj \cR)$ by the isomorphism classes of the projective-injective objects. We refer to Gorsky's work \cite{Gorsky} for the general theory of semi-derived Hall algebras. 
We denote $SH_q^{tw}:= H^{tw}_q[ \sigma(S_i)^{\wedge} , \sigma(\Sigma S_i)^{\wedge}]^{-1}$. 
If $M$ is an object in $\proj \cR$ we denote $[M]$ its class in $SH_q$ and in $SH^{tw}_q$.

\begin{proposition}
Let $\cR$ be as in Theorem \ref{nakajima category}. 
The quantum group $U_{t} (\cg)$ is isomorphic to the twisted semi-derived Hall algebra
 $$SH^{tw}_q/ 
 \langle  [\sigma(S_i)^{\wedge}] [ \sigma(\Sigma S_i)^{\wedge}]-1 \rangle.$$
The Quantum group generators $E_i, F_i, K_i$ correspond under this isomorphism to
\begin{align*}
& E_i \to (q-1) \prod_{j}[ \sigma(S_j)^{\wedge}]^{-  \dim \Hom(\rad P_i, S_j) }[  S_i^{\wedge}]\\
& F_i \to  - (q-1)^{-1} t \prod_{j}   [ \sigma(\Sigma S_j)^{\wedge}]^{-\dim \Hom(\rad P_i, S_j)} [ \Sigma S_i^{\wedge}] \\
& K_i \to [ \sigma(S_i)^{\wedge}] \\
\end{align*}
\end{proposition}
\begin{proof}
By Theorem \ref{nakajima category} we have that  $Comp_{\Z/2}(\proj kQ)$ is equivalent to $\proj \cR$ as exact Frobenius categories. Hence their semi-derived twisted Hall algebras are isomorphic. Now the result follows from Bridgeland \cite{Bridgeland13}.  
\end{proof}

Let us denote by $\gpr_w (\cS)$ the sub-space of $\rep(w, \cS)$ consisting of the $\cS$--modules which are Gorenstein projective. 
Consider the Hall algebra $H(\gpr \cS)$ of constructible functions on the moduli spaces $ \gpr_w \cS$ over $\C$ generated by the indicator functions $\delta_{[M]}$, 
where $M$ runs through all isomorphism classes $[M]$ of $\gpr \cS$. 
Recall that $\delta_{[M]}$ is the constructible function that evaluates to $1$ on every point corresponding to a module isomorphic to $M  \in \gpr \cS$ and evaluates to $0$ everywhere else. 

\begin{theorem}
Specialized at $q=1$, the generic Hall algebra of $\proj \cR$, $H_{q=1}(\proj \cR)$ is isomorphic to the Hall algebra of constructible functions $H(\gpr \cS)$ on $\gpr \cS$. \end{theorem}
\begin{proof}
As shown in Theorem \ref{hall}, for all  $L, N, M \in \proj \cR \cong \gpr \cS$ the Hall number $F^{L}_{M, N}$ is given by a Laurent polynomial $p(q) $ in $q$. Furthermore, we know that $H_q(\gpr \cS) \cong H_q(\proj \cR)$.   Hence $\gpr \cS$ also satisfies the Hall polynomial property.  In Proposition 6.1 of \cite{Reineke06}, Reineke shows that $p(1)$ equals $ \chi(\Ext_{\gpr \cS}^1(M, N)_L)$, the Euler characteristic of the moduli space $ \Ext_{\gpr \cS}^1(M, N)_L$, and hence the two Hall algebras are isomorphic for $q=1$. 

\end{proof}


\section{Hall algebras of quiver varieties}
\subsection{Strata of Gorenstein projective $\cS$--modules}

Let $\cR$ be a $\Hom$-finite Nakajima category associated to a quiver $Q$ of Dynkin type and let $\cS$ be 
the associated singular Nakajima category.

\begin{lemma}\label{phi}
1) The projective module $\sigma(x)^{\wedge} \in \mod \cR$ is bistable for all $x \in C$, hence 
the image of $  \sigma(x)^{\wedge} \in \mod \cS$ vanishes under the stratification functor $CK$. 

2) We have that $K_R \res x^{\wedge} \cong x^{\wedge}$
and $CK(  x^{\wedge}) = x_{ \cp}^{\wedge} $ for all $x \in  \cR_0 - \cS_0$. 
\end{lemma}
\begin{proof}
Clearly, $\sigma(x)^{\wedge}$ is costable. Furthermore we have that  $\Hom(S_z, y^{ \wedge}) \cong \Ext^2( y^{\wedge}, S_{\tau z})$
for all $y \in \cR_0$ and all $z \in   \cR_0 - \cS_0$ by Lemma 3.13 of \cite{Scherotzke}. As the second term vanishes due to the projectivity of $y^{\wedge}$, 
we know that $y^{\wedge}$ is stable. We conclude that $\sigma(x)^{\wedge}$ is bistable. Finally, 
as 
$$K_{LR} ( \sigma(x)^{\wedge}) \cong K_R(  \sigma(x)^{\wedge})\cong  \sigma(x)^{\wedge}$$ we conclude that $CK( \sigma(x)^{\wedge})$ vanishes.

To prove the second part, we note that by the first part $x^{\wedge}$ is stable. Hence there is a monomorphism  $ x^{\wedge} \to K_R \res x^{\wedge}$ whose restriction to $\cS$ is an isomorphism. Furthermore we have that $\Ext^1(S_z, x^{\wedge}) \cong \Ext^1(x^{\wedge} , S_{\tau z}) $ for
all $z\in \cR_0-\cS_0$. The last term vanishes as $x^{\wedge}$ is projective. It follows that $x^{\wedge} \cong K_R \res x^{\wedge}$. 

Next, we observe that there is a canonical surjection 
$$ x^{\wedge} \stackrel{f} \to x_{ \cp}^{\wedge} \to 0.$$ 
Furthermore, the sequence 
$$0 \to \Hom(x_{\cp}^{\wedge}, S_z) \to \Hom(x^{\wedge}, S_z) \to  \Hom(\ker f, S_z) \to \Ext^1(x_{\cp}^{\wedge}, S_z)=0$$
is exact for all $z \in \cR_0-\cS_0$. The first two terms vanish for all $z \not=x$ and for $z=x$ they are one-dimensional. 
Hence $\ker f$ is a maximal costable submodule of $x^{\wedge}$. It follows that $K_{LR}(\res x^{\wedge})$ is given by 
$ \ker f$ and by the definition of $CK$ we have that $CK(\res x^{\wedge})\cong x^{\wedge}_{\cp}$. 
\end{proof}
The following Lemma provides a simple criterion for an $\cS$--module to belong to the exact subcategory $\gpr \cS$. 
\begin{corollary}
An $\cS$--module $M$ lies in $\gpr \cS$ if and only if $K_R M$ is projective. 
\end{corollary}
\begin{proof}
We know by Theorem 5.5 of \cite{Scherotzke}, that every indecomposable module in $\gpr \cS$ is isomorphic to $\res x^{\wedge} $ for an element $x \in \cR_0$. As $K_R \res x^{\wedge} \cong x^{\wedge}$ by Lemma \ref{phi}, this concludes the proof. 
\end{proof}

\begin{lemma} \label{transversal}
The stratum containing $\res  \sigma(S_i)^{\wedge}$ is transversal to the zero stratum and the stratum containing $\res  S_i ^{\wedge} $ is transversal to the stratum containing the simple $\cS$--module $S_{\sigma( \Sigma S_i )} $. 
\end{lemma}
\begin{proof}
The first part follows immediately from the fact that $\sigma(S_i)^{\wedge}$ is bistable by Lemma \ref{phi} 
and therefore $\res \sigma(S_i)^{\wedge}$ is transversal to the zero strata. 
Let $ x \in \cR_0-\cS_0$, then by Lemma 
\ref{phi}, the stratification functor $CK$ takes value $ x^{\wedge}_{ \cp}$ on the stratum containing  $\res x^{\wedge}$.
As $ CK(S_{\sigma S_i} ) =( \Sigma S_i)^{\wedge}$ by Theorem 4.5  \cite{Scherotzke} (note that our convention is slighlty different on the notation of $\sigma(c)$), we recover the claim. 
\end{proof}

\begin{lemma}\label{open}
The irreducible components in $\gpr_w(\cS)$ are open in $\rep(w, \cS)$ for any dimension vector $w$. Furthermore, for every irreducible component $C$, 
there is an $M \in \gpr_w(\cS)$ whose orbit is dense and open in $C$. 
\end{lemma}
\begin{proof}
The tangent vector of $x \in \gpr_w(\cS)$ are given by all $x +\epsilon D \in \mod (\cS\otimes k[\epsilon]/\epsilon^2)$ where $D$ is a derivation. This tangent vector lies in $\gpr (\cS\otimes k[\epsilon]/\epsilon^2)$  if and only if $\Ext^i( x+ \epsilon D, \cS\otimes k[\epsilon]/\epsilon^2) $ vanishes for all $i$. This condition is always satisfied if $ 
x\in \gpr \cS$. Hence the tangent space of $x$ with respect to $\rep(w, \cS)$ coincides with the tangent space with respect to $\gpr_w \cS$. As a consequence, $\gpr_w \cS$ is open.

Next, we have that every indecomposable object of $\gpr \cS$ is rigid, hence its orbit is open and dense. Let $M \in \gpr_w(\cS)$ have the property that $end(M)$ is minimal. We prove by contradiction that $M$ is rigid.  Assume that $M$ decomposes as
direct sum $M_1 \oplus M_2$ such that $ \Ext^1(M_1, M_2)$ does not vanish. As a consequence, there is a non-split exact sequence $0\to M_2 \to N \to M_1 \to 0$ with $N \in \rep(w, \cS)$. As $\gpr \cS$ is closed under extensions, we also have $N \in \gpr \cS$. Furthermore, $end(N)< end(M)$ which is a contradiction. Hence $M$ is rigid. It follows that $\gpr_w(\cS)$ is either empty or contains a dense open orbit. 
\end{proof}

Note that the next result does not hold in general if either $A$ or $B$ do not lie in $\gpr \cS$. 

\begin{proposition}
Let $A \in \gpr \cS$ be indecomposable and non-projective. Suppose that $B$ is in $\gpr \cS$ and that $A$ and $B$ lie in the same stratum. 
Then $A $ and $B$ are isomorphic. 
\end{proposition}
\begin{proof}
As $A$ and $B$ lie in the same stratum, we have that $CK(A) \cong CK(B)$ and $K_R(A)$ and $K_R(B)$ are both projective. Furthermore as $CK(A)$ is indecomposable, so is $CK(B)$.
Hence, we obtain the following commutative diagram
\[ \xymatrix{ 
0 \ar[r] & K_{LR}(A) \ar[r] & K_R(A) \ar[r] & CK(A)  \ar[r] &0 \\
0 \ar[r] & K_{LR}(B ) \ar[r]  \ar@{->>}[u] & K_R(B) \ar[r] \ar@{->>}[u] & CK(B) \ar@{=} [u]  \ar[r] &0. 
}\]
As $K_{LR}(A)$ and$K_{LR}(B)$ have the same dimension vector, they are necessarily isomorphic. 
\end{proof}

\subsection{The graded Grothendieck ring of quiver varieties}

In this section, we introduce the graded Grothendieck ring associated to generalized quiver varieties. 
We show that the dual graded Grothendieck ring has an algebra structure induced by the `Hall' convolution 
product on the level of the derived category of constructible sheaves on quiver varieties. 

For $\alpha, \beta$ two dimension vectors of $\cS$, we fix $W_0 \subset V_{\alpha+\beta}$ a vector space of graded dimension $\alpha$ and let $F_{\alpha, \beta}$ be the subvariety 
$$F_{\alpha, \beta}:= \{ y \in \rep(\alpha+\beta, \cS)\text{ such that } y(W_0) \subset  W_0 \} \subset \rep(\alpha+\beta, \cS)$$
We consider the convolution diagram, where $p$ is the projection on the submodule $y|_{W_0}$ and quotient $y/ y|_{W_0}$ and $q$ is the embedding. 
\[
 \xymatrix{ \rep(\alpha, \cS) \times \rep(\beta, \cS ) &   F_{\alpha, \beta} \ar[l]^{\ \ \  \ \ \ \ \ \ \ \ \ p}  \ar[r]_{q  \ \ \ \ \ } &\rep(\alpha+\beta, \cS )}
 \]
 
Let us denote by $D(\rep(\alpha, \cS))$ the derived category of constructible sheaves on $\rep(\alpha, \cS)$ and
by $$\Delta: D(\rep(\alpha+\beta, \cS )) \to D(\rep(\alpha, \cS)) \times D(\rep(\beta, \cS ) ), F \mapsto p_!q^* F$$ 
the comultiplication. 

Recall that we have a proper map 
$\pi: \cm(v, w) \rightarrow \cm_0(w) \cong \rep (w, \cS)$.  We denote $\pi(v,w) \in D(\rep(w, \cS))$ the pushforward along $\pi$ of the constant sheaf $\mathbb{C}_{\cm(v,w)}$. By the Decomposition Theorem  $\pi(v,w)$ is a direct sum of shifts of intersection cohomology complexes. 

Based on \cite{Scherotzke}, we can use an analogous proof to \cite{VargnoloVasserot03} to recover the following result in the case of generalized quiver varieties.

\begin{proposition}\label{pi}
The comultiplication yields 
$$\Delta (\pi(v,w) )=\bigoplus_{v_1+v_2 =v, w_1 +w_2=w}\Sigma^{d(v_2, w_2,v_1, w_1)  -   d(v_1, w_1,v_2, w_2)}  \pi(v_1, w_2) \otimes \pi(v_2, w_2).$$ As a consequence all
direct summands of $\pi(v,w)$ are sent  to the external product of direct sums of shifts of IC-sheaves appearing as direct summands of sheaves $\pi(v,w)$.
\end{proposition} 
Here the coefficient $d(v_1, w_1,v_2, w_2)$ is defined as in \cite{VargnoloVasserot03} by taking the Cartan matrix of the generalized quiver variety of \cite{Scherotzke} Section 4.

The previous result shows that the IC complexes appearing as direct summands of $\pi(v,w)$ are closed under $\Delta$.  We define the 
\emph{graded Grothendieck ring}, denoted here by $K^{gr}(\mod \cS)$, to be the free $\Z[t, t^{-1}]$--module with basis the IC sheaves appearing as direct summands of $\pi(v,w)$ for all $(v,w)$. Interpreting $t$ as the shift functor, the previous result 
shows that $\Delta$ induces a comultiplication structure on $K^{gr}(\mod \cS)$. Hence its $\Z[t, t^{-1}]$ dual $K^{gr *}(\mod \cS):= \Hom_{\Z[t, t^{-1}]}( K^{gr}(\mod \cS), \Z[t,t^{-1}])$ 
becomes a $\Z[t, t^{-1}]$--algebra. 

Let  $\cl(v,w)$ be the intersection cohomology complex associated to 
the stratum $\cm_0^{reg}(v,w) \subset \cm_0(w)$ with respect to the trivial local system. As $\cl(v,w)$ appears as a direct summand of $\pi(v,w)$,   the IC complex 
$\cl(v,w)$ is a generator of the free module $K^{gr}(\mod \cS)$. We denote $L(v,w)$ the corresponding element of the dual set of generators of $K^{gr *}(\mod \cS)$. 

\begin{lemma}\label{compatibility}
Let $i: \gpr \cS \to \mod \cS$ be the canonical embedding, then $i^*$ commutes with $\Delta$. As a consequence, the direct sum of shifted IC-sheaves appearing as direct summands of $i^* \pi(v,w)$ are mapped to direct sums of shifted $IC$-sheaves by $\Delta$. 
\end{lemma}
\begin{proof}
As $\gpr_w( \cS)$ is an open variety in $\rep(w,\cS)$ by Lemma \ref{open}, the embedding $i$ is a smooth map and $i^*$ maps IC  complexes to IC complexes. Hence, $i^*\pi(v,w)$ is a direct sum of shifts of IC complexes. Let 
$$F_{\alpha, \beta}(\gpr \cS) := \{ y \in  F_{\alpha, \beta} \text{ and } y|_{W_0}, y/ y|_{W_0}  \in \gpr \cS \}$$
 be the sub variety of $F_{\alpha, \beta}$. 
 Then we obtain the following convolution diagram, where the left hand side square is a fibre product, as $\gpr \cS$ is closed under extension  
\[
\xymatrix{ \gpr_{\alpha} \cS \times \gpr_{\beta} \cS \ar[d]_{i \times i}    &   F_{\alpha, \beta}( \gpr \cS)  \ar[l]^{p} \ar[d]_{i}  \ar[r]_{q}& \gpr_{\alpha+\beta} \cS \ar[d]_i \\
 \rep(\alpha, \cS) \times \rep(\beta, \cS ) &   F_{\alpha, \beta} \ar[l]^{\ \ \  \ \  \ p}  \ar[r]_{q \ \ \ } &\rep(\alpha+\beta, \cS )}
\]
We find that $$\Delta(i^* F)=p_! q^{*} (i^* F)= p_! i^* q^*F = (i \times i)^*p_! q^{*} F=(i \times i)^* \Delta (F)$$ for all $F \in D^b(\gpr_{\alpha+ \beta} \cS)$ by smooth base change. This shows that $\Delta$ commutes with $i^*$ and concludes the proof.  

\end{proof}

\section{Geometric realizations of the quantum group}

Let now $\cS$ be the singular Nakajima category that we considered in Example \ref{exa2}. That is, $F=\Sigma^2$ and the configuration consists of the simple $kQ$--modules $S_i$ and their shifts $\Sigma S_i$ seen as objects of $\cp \cong \cd_Q/  \Sigma^2$.  We showed in Theorem  \ref{nakajima category} that with this choice of configuration
$ \gpr( \cS)$ is equivalent to the category $Comp_{\Z/ 2 }(\proj \F_q Q )$ of two-cyclic complexes considered by Bridgeland in \cite{Bridgeland13}. 
Bridgeland proves that the quantum group $U_t(\cg)$ is isomorphic to the twisted semi-derived Hall algebra of $Comp_{\Z/ 2 }(\proj \F_q Q )$. 

In this Section we explain two alternative approaches to the quantum group that are closer in spirit to Lusztig's work \cite{Lusztig90}. One is due to Qin \cite{Qin14}, the other is our own. By Theorem \ref{iso} our construction is isomorphic to Qin's, but it has the important advantage of allowing a direct comparison with Bridgeland's semi-derived Hall algebra. We start by explaining our construction, and then will recall Qin's. 

\subsection{A new geometric construction of the quantum group} 

Let $M$ be an object 
 in $Comp_{\Z/ 2 }(\proj \F_q Q ) \cong \gpr \cS$. Denote 
 $\mathcal{L}(M)$ the IC complex on the stratum of the Nakajaima quiver variety that contains $M$. The IC complex $\cl(M)$ defines an element of $K^{gr }(\gpr \cS)$. We denote $L(M)$ the dual element in the dual graded Grothendieck ring, $L(M) \in K^{gr *}(\gpr \cS)$. 
%
Bridgeland's semi-derived Hall algebra is a localization of the  Hall algebra of $Comp_{\Z/ 2 }(\proj \F_q Q )$, and we shall  
localize $K^{gr *}(\gpr \cS)$ 
in a similar way. Namely, consider the ring 
$$K^{gr *}(\gpr \cS) [L(\sigma(S_i)^{\wedge}) ,L(\sigma(\Sigma S_i)^{\wedge} )]^{-1}.$$ 
The assignment $$M \in Comp_{\Z/ 2 }(\proj \F_q Q ) \mapsto L(M) \in K^{gr *}(\gpr \cS)$$ extends to  a linear map 
$$
L: SH_q \rightarrow K^{gr *}(\gpr \cS) [L(\sigma(S_i)^{\wedge}) ,L(\sigma(\Sigma S_i)^{\wedge} )]^{-1}.$$

It will follow from Theorem \ref{iso} that this map is actually compatible with the product structures. 
Let $E_i, F_i , K_i \in SH_q$ be the elements corresponding to the generators of the Quantum group, see \cite{Bridgeland13}. 
We focus on the sub-algebra of  
$K^{gr *}(\gpr \cS) [L(\sigma(S_i)^{\wedge}) ,L(\sigma(\Sigma S_i)^{\wedge} )]^{-1}$ generated by  $L(E_i), L(F_i) , L(K_i)$ and 
$L(K_i^{-1})$: 


\begin{itemize}\label{generator of QG}
\item 
$E_i = (q-1) \prod_{j}[ \sigma(S_j)^{\wedge}]^{- a_{ij} }[  S_i^{\wedge}]    \mapsto L(E_i) =$\\ $(q-1) \prod_{j}L(\sigma(S_j)^{\wedge})^{-  a_{ij}}L(S_i^{\wedge})$
\item $F_i =(q-1)^{-1}t \prod_{j}   [ \sigma(\Sigma S_j)^{\wedge}]^{-a_{ij}} [ \Sigma S_i^{\wedge}]  \mapsto L(F_i)  =$ \\ $= (q-1)^{-1}t \prod_{j}   L( \sigma(\Sigma S_j)^{\wedge})^{-a_{ij}} L( \Sigma S_i^{\wedge}),$ 
\item $K_i = [ \sigma(S_i)^{\wedge}]  \mapsto L(K_i) =  L( \sigma(S_i)^{\wedge})$ 
\item $  K_i^{-1} =  [ \sigma(S_i)^{\wedge}]   \mapsto L(K_i^{-1}) =  L(\sigma(\Sigma S_i)^{\wedge}).$
\end{itemize}

Here $a_{ij}:= \dim \Ext^1(S_i, S_j)=\dim \Hom(\rad P_i, S_j)$.   We denote this subalgebra 
$\cS \ch_{\gpr \cS}$. It is a consequence of Theorem \ref{iso} that  
$\cS \ch_{\gpr \cS}$, after twisting, recovers the  quantum group. 

\subsection{Qin's construction of the quantum group} 
\label{sec:qin}
We give a brief survey  of Qin's results. We refer to \cite{Qin14} for more details. Instead of 
$K^{gr *}(\gpr \cS)$, Qin works with $K^{gr *}(\mod \cS)$ and its localizations.  
If $M$ is in $\mod \cS$, let $\cl(M)$ be the IC complex associated to the stratum containing $M$, $\cl(M) \in K^{gr}(\mod \cS)$. 
Denote $L(M) \in K^{gr*}(\mod \cS)$ the dual of $\cl(M)$. Let $\cl(v_i,w_i)$ and $\cl(v'_i,w_i)$ be the IC complexes on $\cm_0^{reg}(v_i, w_i)$ and $\cm_0^{reg}(v_i', w_i)$,  where   
$$ w_i:=e_{\sigma(S_i)}+e_{\sigma(\Sigma S_i)}, \quad v_i(z):= \dim {\cp} (z, S_i), \quad v_i'(z):= \dim {\cp} (z ,\Sigma S_i)$$ 
for all $z \in \cp_0$. 
Denote $L(v_i,w_i)$ and $L(v'_i,w_i)$ their duals in  $K^{gr *}(\mod \cS)$. 
We refer to \cite{Qin14} for the definition $\tilde U_t(\cg)$, and its presentation in terms of the generators 
$E_i$, $F_i$, $H_i$ and 
$H_i'$. Recall that, setting $H_i^{-1}=H_i'$ and $H_i = K_i$, one recovers the quantum group $U_t(\cg)$.

\begin{remark}
Qin's starting point are all cyclic quiver varieties. These are obtained as the moduli spaces of a singular Nakajima category with a configuration consisting of all $\cp$--modules. 
However the generators $L(v,w)$ that Qin considers have support $w$ only in the simple $kQ$--modules $S_i$ and their one shifts $\Sigma S_i$. Hence we can restrict to the cyclic quiver varieties realized as moduli spaces of $\cS$. 
\end{remark}


\begin{theorem}\cite{Qin14}\label{qin}
There is an algebra embedding
$$\tilde U_t(\cg)[H_i, H_i']^{-1} \hookrightarrow  K^{gr *}(\mod \cS)^{tw}[L(v_i,w_i) ,L(v'_i,w_i)]^{-1}$$ induced by 


\begin{itemize}

\item $E_i \mapsto L(S_{\sigma S_i})$
\item $F_i \mapsto L(S_{\sigma \Sigma S_i})$ 
\item $H_i \mapsto L(v_i,w_i)$
\item $H_i' \mapsto L(v_i',w_i)$.
\end{itemize}

\end{theorem}

\begin{definition}
Denote $\cS \ch_{\mod \cS}$ the sub-algebra of  
$$
K^{gr *}(\mod  \cS)^{tw}[L(v_i,w_i) ,
L(v'_i,w_i)]^{-1}
$$ given by the image of  $\tilde U_t(\cg)[H_i, H_i']^{-1}$. 
\end{definition}

\subsection{The proof of the main theorem}  
Let $\cm^{reg}_0(v_i, w_i)$ and $\cm^{reg}_0(v_i', w_i)$ be the strata we considered in Section in \ref{sec:qin}. By \cite{Qin14} they are transversal to the zero stratum. Therefore Lemma \ref{transversal} yields the following result.

\begin{corollary}\label{transversal generator}
The strata containing $\sigma(S_i)^{\wedge}$ and $\sigma(\Sigma S_i)^{\wedge}$ 
are 
transversal to $\cm^{reg}_0(v_i, w_i)$ and $\cm^{reg}_0(v_i', w_i)$. 
Furthermore, the strata containing $\res S_i^{\wedge} $ and $\res \Sigma S_i^{\wedge}$  
 are transversal to the strata containing 
$S_{\sigma \Sigma S_i}$ and $S_{\sigma S_i}$.  
\end{corollary}

\begin{lemma}\label{generators}
There is a non-split exact sequence 

$$ 0\to P \to \res   S_i ^{\wedge}  \to S_{\sigma(\Sigma^{-1} S_i)} \to 0 ,$$
where $P$ denotes the projective $\cS$--module $\bigoplus_{c \in C-\Sigma^{-1} S_i}  \cp( S_i, \Sigma c) \otimes \sigma(c)^{\wedge}$.
\end{lemma} 

\begin{proof}
The claim follows from the  commutative diagram below, where the last two rows are exact sequences by Lemma 3.13 and 3.14 of  \cite{Scherotzke}. 
\[ 
\xymatrix{    0 \ar[d] \ar[r] & P  \ar[r]_{\id} \ar[d] & P \ar[d]  \\
 \res \Sigma^{-1} S_i ^{\wedge} \ar[r] \ar[d]_{\id}  & P \oplus  \sigma^{\wedge}(\Sigma^{-1} S_i)  \ar[r] \ar[d] & \res  S_i ^{\wedge} \ar[d] \\
 \res \Sigma^{-1} S_i^{\wedge}  \ar[r] &  \sigma^{\wedge}(\Sigma^{-1} S_i)  \ar[r] & S_{\sigma( \Sigma^{-1} S_i)} . 
}\]
The rightmost column is a short exact sequence by the Snake lemma.  
\end{proof} 
By the transversal slice Theorem 3.3.2 of  \cite{Nakajima01}, we know that the direct summands appearing in $\pi(v,w)$  are 
the shifts of sheaves $\cl(v',w)$ with $v' \le v$. 
Combining this with Proposition \ref{pi} we know that $$L(v,w) L(v',w') = \sum_{v'' \ge v+v'} a_{v''} L(v'', w+w')$$ where 
$a_{v''} \in \N[t^{\pm}]$ and $a_{v_1+v_2}= t^{d(v_2, w_2,v_1, w_1)  -   d(v_1, w_1,v_2, w_2)}$.

Note that, up to isomorphism, there exists exactly one indecomposable $\cR$--module with dimension vector $(v_i, w_i)$ and $(v_i',w_i)$ respectively. We will denote these modules   $M_i$ and $M_i'$ respectively.

\begin{lemma}\label{cartan}
The modules $\sigma(S_i)^{\wedge}$ have a filtration by 
$ M_{i_1}, \cdots, M_{i_n}$,  corresponding to the order of the simple $\C Q$--modules $S_{i_1}, \ldots , S_{i_n}$ 
appearing in a composition series of the projective module $P_i$. 
It follows that $$ L(v_{i_1}, w_{i_1}) \cdots  L(v_{i_n}, w_{i_n}) =L( \sigma(S_i)^{\wedge}).$$ 
\end{lemma}
\begin{proof} 
By Theorem \ref{nakajima category} 
the category $\gpr \cS$ is  equivalent to the category of $2$-cyclic complexes of projective $\C Q$--modules. The objects $\sigma(S_i)^{\wedge}$ and 
$\sigma(\Sigma S_i)^{\wedge}$ 
correspond under this equivalence 
to $P_i \stackrel{\id}  \rightarrow P_i$ and $P_i \stackrel{\id}  \leftarrow P_i$ respectively.   Recall that $\sigma(S_i)^{\wedge}$ is bistable and $K_{LR} \sigma(S_i)^{\wedge} \cong \sigma(S_i)^{\wedge}$ by Lemma \ref{phi}. Hence using Lemma \ref{lemma:stable}, we know that $\sigma(S_i)^{\wedge}$ lies in the stratum with dimension vector $\gdim \sigma(S_i)^{\wedge}$. 
Hence we have that $$\sigma(S_i)^{\wedge}(S_j)= \Hom(\sigma(S_j)^{\wedge}, \sigma(S_i)^{\wedge}) \cong
 \Hom(P_j, P_i)= \sigma(S_i)^{\wedge}(\Sigma S_j)$$ for all $i, j$. This proves the first part of the statement. Now by \cite{Qin14}
 $$L(v_i, w_i) \cdot L(v_{i_1}, w_{i_1}) \cdots  L(v_{i_n}, w_{i_n}) =L( \sum v_{i_k}, \sum w_{i_k})$$ and as $\sigma(x)^{\wedge}$ can be obtained by iterated extensions of the $M_i$, we have that
 $v\ge \sum v_{i_k}$ and $w=\sum w_{i_k}$, where $(v,w)$ is the dimension vector of $K_{LR} \sigma(x)^{\wedge}=\sigma(x)^{\wedge}$ in $\mod \cR$. We know by Lemma \ref{transversal} that $(v,w)$ vanishes under $CK$ and that the same holds for $(\sum_k v_{i_k}, w)$. Now as the Cartan matrix is injective when restricted to a dimension vector of an indecomposable object in $\gpr \cS$, we obtain by Theorem of \cite{Scherotzke} that $\sum v_{i_k}=v$. Hence by \cite{VargnoloVasserot03}, we have that $L(v_i, w_i) \cdot L(v_{i_1}, w_{i_1}) \cdots  L(v_{i_n}, w_{i_n})=L(\sigma(x)^{\wedge}) $. 
This finishes the proof.
\end{proof} 

\begin{theorem}\label{iso}
The pullback induces an isomorphism of algebras 
$$\cS \ch_{\gpr \cS}
\cong 
\cS \ch_{\mod \cS}.
$$ 
Further under this isomorphism the quantum group generators $E_i$, $F_i$ are mapped to each other. 
\end{theorem}
\begin{proof}
By Lemma \ref{compatibility}, the embedding of $\gpr \cS$ into $\mod \cS$ induces a surjective co-algebra homomorphism $$K^{gr}(\mod \cS) \to K^{gr}(\gpr \cS), \cl(v,w) \mapsto i^* \cl(v,w)$$ and hence dualizing yields an injective algebra homomorphism 
$$\phi: K^{gr *}(\gpr \cS) \to K^{gr *}(\mod \cS), L(v,w) \mapsto L(v,w).$$ Furthermore, we have that $L(\sigma(x)^{\wedge})$ is by Lemma \ref{generators} a product of objects $L(v_i, w_i)$ and 
$L(\sigma(\Sigma(x))$ is a product of the associated objects $L( v'_i, w_i)$. Hence $L(\sigma(S_i)^{\wedge})$ and $ L(\sigma(\Sigma S_i)^{\wedge} )$ are invertible in 
$$K^{gr *}(\mod \cS)[L(v_i, w_i), L(v_i', w_i)]^{-1}.$$ Thus, after localization, we get an algebra monomorphism 
$$K^{gr *}(\gpr \cS)[L(\sigma(S_i)) , L(\sigma(\Sigma S_i))]^{-1}  \to K^{gr *}(\mod \cS)[L(v_i, w_i), L(v_i', w_i)]^{-1}.$$ Furthermore, as the indecomposable projective $\C Q$--modules $P_i$ form a basis of the Grothendieck ring of $\mod \C Q$, we know by Lemma \ref{cartan} that the 
subalgebra generated by the elements $L(\sigma(S_i))^{\pm} , L(\sigma(\Sigma S_i))^{\pm}$ is isomorphic to the subalgebra generated by the elements $L(v_i, w_i)^{\pm}, L(v_i', w_i)^{\pm}$. 
By Lemma \ref{generators}, there is a non-split exact sequence in $\mod \cS$ given by 
$$ 0\to P \to \res  S_i ^{\wedge}  \to S_{\sigma(\Sigma^{-1} S_i)} \to 0 .$$ 
Furthermore, we have that 
$$\gdim K_{LR} ( \res  S_i ^{\wedge} )=\gdim K_{LR} P + \gdim K_{LR} S_{\sigma(\Sigma^{-1} S_i)}=:(v,w)$$ combining \ref{transversal} and Proposition 4.6 of \cite{Scherotzke}. Now all strata $\cm_0^{reg}(v',w)$ such that $v' > v$ are empty. We conclude that
 $$L( S_{\sigma(\Sigma^{-1} S_i)})L(P) = a_i L(\res  S_i ^{\wedge} )$$ for some coefficient $a_i$ which is a power of $t$ and therefore 
 $$L( S_{\sigma(\Sigma^{-1} S_i)})  = a_i L(\res S_i ^{\wedge} )   L(P)^{-1}.$$  
Hence all generators of the algebra considered by Qin are indeed in the image of $\phi$.
Therefore we find that the IC-sheaves associated to the isoclasses of modules used by Bridgeland are indeed isomorphic under $\phi$ to the IC-sheaves generating the subalgebra of $K^{gr *}(\mod \cS)$ constructed by Qin. 
\end{proof}

\def\cprime{$'$} \def\cprime{$'$}
\providecommand{\bysame}{\leavevmode\hbox to3em{\hrulefill}\thinspace}
\providecommand{\MR}{\relax\ifhmode\unskip\space\fi MR }
\providecommand{\MRhref}[2]{%
  \href{http://www.ams.org/mathscinet-getitem?mr=#1}{#2}
}
\providecommand{\href}[2]{#2}

\end{document}